\documentclass[11pt]{article}
\usepackage[utf8]{inputenc}
\usepackage[OT1]{fontenc}
\usepackage[shortlabels]{enumitem}
\usepackage{amsmath,amsfonts,amssymb,amsthm,mathtools,bm}
\usepackage{xcolor}
\usepackage{soul}
\usepackage{thmtools}
\usepackage{thm-restate}
\usepackage[noblocks,affil-it]{authblk}
\usepackage[mathscr]{euscript}
\usepackage[margin=3cm]{geometry}
\usepackage{hyperref}
\usepackage{todonotes}
\usepackage{{nicematrix040623}}
\usepackage{pgfplots}
\pgfplotsset{compat=newest}

\newcommand{\R}{{\mathbb{R}}}
\newcommand{\Z}{{\mathbb{Z}}}
\newcommand{\bb}{{\mathscr{B}}}
\newcommand{\aaa}{{\mathscr{A}}}
\newcommand{\sss}{{\mathscr{S}}}
\newcommand{\eps}{{\varepsilon}}
\newcommand{\zero}{\mathbf{0}}
\newcommand{\la}{\langle}
\newcommand{\ra}{\rangle}

\declaretheorem{theorem}
\declaretheorem{lemma}

\declaretheorem{example}
\declaretheorem{conjecture}

\title{\textbf{Stability for binary scalar products}}
\author[1,2]{Andrey Kupavskii}
\affil[1]{Moscow Institute of Physics and Technology, Moscow, Russia}
\affil[2]{G-SCOP, CNRS, Grenoble, France}
\author[3]{Dmitry Tsarev}
\affil[3]{University of Cambridge, Cambridge, United Kingdom}
\date{}

\begin{document}

\maketitle

\begin{abstract}
    Bohn, Faenza, Fiorini, Fisikopoulos, Macchia, and Pashkovich (2015) conjectured that 2-level polytopes cannot simultaneously have many vertices and many facets, namely, that the maximum of the product of the number of vertices and facets is attained on the cube and cross-polytope. This was proved in a recent work by Kupavskii and Weltge. In this paper, we resolve a strong version of the conjecture by Bohn et al., and find the maximum possible product of the number of vertices and the number of facets in a 2-level polytope that is not affinely isomorphic to the cube or the cross-polytope. To do this, we get a sharp  stability result of Kupavskii and Weltge's upper bound on $\left|\aaa\right|\cdot\left|\bb\right|$ for $\aaa,\bb \subseteq \R^d$ with a property that $\forall a \in \aaa, b \in \bb$ the scalar product $\la a, b\ra \in\{0,1\}$.
\end{abstract}

\section{Introduction}

A polytope $P$ is {\it 2-level} if for every facet-defining hyperplane $H$ there is a parallel hyperplane $H'$ such that $H \cup H'$ contains all vertices of $P$. Basic examples of 2-level polytopes are simplices, hypercubes and cross-polytopes, but they also generalize a variety of interesting polytopes such as Birkhoff, Hanner, and Hansen polytopes, order polytopes and chain polytopes of posets, stable matching polytopes, and stable set polytopes of perfect graphs~\cite{aprile18}. Combinatorial structure of two-level polytopes has also been studied in~\cite{fiorini16}, and enumeration of such polytopes in~\cite{bohn18} led to a beautiful conjecture about their vertex and facet count, which was proven in~\cite{kupavskii22}:
\begin{restatable}{theorem}{twoLevelOld}
    \label{two_level_old_bound}
    If $P$ is a $d$-dimensional 2-level polytope, its number of vertices $f_0(P)$ and facets $f_{d-1}(P)$ satisfy
    \[
        f_0(P) \cdot f_{d-1}(P) \leq d 2^{d+1}.
    \]
\end{restatable}
\noindent This bound is tight, as is witnessed by polytopes that are affinely isomorphic to the cube or the cross-polytope. Authors of~\cite{bohn18} conjectured that those are the only instances where equality is attained (see \cite{aprile18}). In this paper, we prove this in a strong sense:

\begin{restatable}{theorem}{twoLevelNew}
    \label{two_level_new_bound}
    Fix $d>1$. Let $P$ be a $d$-dimensional $2$-level polytope that is not affinely isomorphic to the cube or the cross-polytope. Then 
    \[
        f_0(P) \cdot f_{d-1}(P) \leq \left(d-1\right) 2^{d+1} + 8\left(d-1\right).
    \]
\end{restatable}
 \noindent The following two examples demonstrate tightness of the bound in Theorem \ref{two_level_new_bound}.
\begin{example}[Suspension of a cube]\label{CubeSuspension}
    Let $\{e_i\}$ be the standard basis of $\mathbb{R}^d$,
    \begin{equation*}
        P=\operatorname{Conv}\Biggl(\left\{\sum_{i=1}^{d-1}\varepsilon_i e_i: \varepsilon_i\in\{-1, 1\} \text{ for }i\in [d-1]\right\}\cup\left\{e_d, -e_d\right\}\Biggr).
    \end{equation*}
 Here $f_0(P) = 2 + 2^{d-1}$ and $f_{d-1}(P) = 4(d-1)$.
\end{example}
\begin{example}[Cross-polytope $\times$ segment]\label{OctahedronCrossSegment}
    Let $\{e_i\}$ be the standard basis of $\mathbb{R}^d$,
    \begin{equation*}
        P=\operatorname{Conv}\bigl(\left\{ \eps_i e_i+\eps_d e_d: i\in [d-1], \eps_i,\eps_d\in \{-1,1\}\right\}\bigr).
    \end{equation*}
      This is (up to coordinate scaling) the dual of the polytope in the previous example and, in particular,  $f_0(P) = 4(d-1)$ and $f_{d-1}(P) = 2 + 2^{d-1}$. 
\end{example}

\noindent As in the paper \cite{kupavskii22}, the main intermediate result that is of independent interest concerns families of vectors with binary scalar products.

\begin{restatable}[]{theorem}{mainth}\label{d2d_plus_2d}
     Let $\aaa,\bb \subseteq \R^d$ be families of vectors that both linearly span $\R^d$. Suppose that $\la a, b\ra \in \{0,1\}$ holds for all $a \in \aaa$, $b \in \bb$. Furthermore, suppose that $|\aaa|,|\bb|\ge d+2$. Then \begin{equation}\label{eqmainvec}\left|\mathcal{A}\right| \cdot\left|\mathcal{B}\right| \leq d 2^d + 2d.\end{equation}
\end{restatable}

\noindent This theorem is in fact a tight stability result for the following theorem, which was the main result of~\cite{kupavskii22}.

\begin{restatable}{theorem}{oldMainTh}
    \label{d_plus_one_two_d}
    Let $\aaa,\bb \subseteq \R^d$ both linearly span $\R^d$ such that $\la a, b\ra \in \{0,1\}$ holds for all $a \in \aaa$, $b \in \bb$.
    Then we have $|\aaa| \cdot |\bb| \le (d+1) 2^d$.
\end{restatable}

\noindent We give two examples that demonstrate tightness of the bound in Theorem \ref{d2d_plus_2d}.

\begin{example}\label{cubeOctop}
    Let $\{e_i\}$ be the standard basis of $\R^d$, 
    \begin{align*}
        \aaa=&\left\{\sum_{i=2}^{d}\delta_i e_i: \delta_i\in \{0,1\} \text{ for all }i\in [2,d]\right\}\cup\left\{e_1\right\},\\ 
        \bb=&\big\{\delta_1 e_1 + e_j: j\in [2,d] \text{ and }\delta_1\in \{0,1\}\big\} \cup \left\{e_1, 0\right\}.
    \end{align*}
    Here $\left|\mathcal{A}\right|=2^{d-1}+1$ and $\left|\mathcal{B}\right|=2d$.
\end{example}
The example above has both subsets within the binary cube, thus it can be interpreted as two families of sets $\aaa$ and $\bb$ such that $\forall A\in \aaa,\,B\in\bb$ we have $|\aaa\cap\bb|\in\{0,1\}$. This does not hold for the example below.
\begin{example}\label{crosspoly}
    Let $\{e_i\}$ be the standard basis of $\mathbb{R}^d$,
    \begin{align*}
        \aaa=&\left\{e_d+\sum_{i=1}^{d-1}\varepsilon_i e_i: \eps_i \in \{-1,1\} \text{ for all }i\in [d-1]\right\}\cup\left\{0\right\},\\
        \bb=&\left\{\frac{1}{2}\left(e_d+\varepsilon_i e_i\right): i\in [d], \eps_i\in \{-1,1\}\right\}.
    \end{align*}
    As in Example \ref{cubeOctop}, $\left|\aaa\right|=2^{d-1}+1$ and $\left|\bb\right|=2d$.
\end{example}

\paragraph{Outline}
The proof of Theorem~\ref{d2d_plus_2d} builds on the proof of Theorem~\ref{d_plus_one_two_d}, thus, in the next section we present the necessary claims and inequalities from \cite{kupavskii22}. In Section~\ref{sec31} we prove a baby variant of Theorem~\ref{d2d_plus_2d}, that gives uniqueness of the extremal example for Theorem~\ref{d_plus_one_two_d}. The structure of this proof is then reused in Section~\ref{sec32}, where we prove Theorem~\ref{d2d_plus_2d}. In Section~\ref{sec2level} we prove our main result, Theorem~\ref{two_level_new_bound}.
Proofs of claims from \cite{kupavskii22} are provided in Appendix~\ref{appendix} to make this paper is self-contained.

\paragraph{Discussion of the proofs} The proofs of our main results build on the proofs from \cite{kupavskii22}, but require several new ingredients, both combinatorial and, most importantly, geometric. The general idea is to project our families onto a certain subspace and make use of induction on the dimension of the ambient space. Unfortunately, there is quite a bit of case analysis involved, one reason being that there are actually many different configurations that are close to the bound in Theorem~\ref{d2d_plus_2d}. This is witnessed by some of our computer enumeration results below and by explicit constructions that are similar in spirit to Examples~\ref{cubeOctop} and~\ref{crosspoly}, for instance, Example~\ref{generalCubeOctop} discussed below. Filtering all of them out requires different considerations. Another reason is that Theorem~\ref{d2d_plus_2d} has a condition on the sizes of $\aaa$ and $\bb$, and thus before invoking induction hypothesis we must deal with the cases where one of the projected families is small. Geometrically, the most interesting cases are: 3c in the proof of Theorem~\ref{d2d_plus_2d}, where $\aaa$ and $\bb$ switch roles, and we have to study a projection onto a certain subspace formed by vectors of $\bb$, followed by adding a twist on the choice of the vector $b_d$, along which we project in order to use induction; the last case in the proof of Theorem~\ref{two_level_new_bound}, in which we reveal the exact geometric structure of $\aaa$ by reducing the problem to a simple question about families of subsets of $[d-1]$ with small pairwise differences.

By utilising some observations made at the beginning of Section~\ref{secStability} we were able to enumerate all families with binary scalar products (up to linear isomorphism of individual sets) in dimensions $d \leq 5$. The maximal (with respect to the product order on $\mathbb{N}\times\mathbb{N}$) pairs $(|\aaa|,|\bb|)$ for \mbox{$4$-dimensional} families are $(5, 16),\,(6, 12),\,(7, 10),\,(8, 9),\,(9, 8),\,(10, 7),\,(12, 6)$ and $(16, 5)$, which, together with examples above, demonstrates that the bound of $d 2^d + 2d$ might be achieved on sets of different structures (pairs $(6,12), (8,9), (9,8)$ and $(12,6)$ above). Figure~\ref{posSizesPlot} depicts all possible sizes of families in $\R^5$, Figure~\ref{pos-Min-Prod-Plot} shows the same data plotted with $|\aaa|\cdot|\bb|$ against $\operatorname{min}(|\aaa|,|\bb|)$ for clarity. 

\begin{figure}[!h]
\centering
\begin{tikzpicture}
\begin{axis}[
    xlabel={$|\aaa|$},
    ylabel={$|\bb|$},
    axis lines=left,
    xmin=3, xmax=34,
    ymin=3, ymax=34, 
    legend pos=north east,
    legend style={font=\scriptsize},
    width=0.5\textwidth,
    height=0.5\textwidth,
    label style={font=\small}
]
    
\addplot[domain=5:34, samples=100, color=black, dashed] {192/x};
\addlegendentry{$x\cdot y = (5+1)2^5$}

\addplot[only marks, color=black, mark=*, mark options={fill=black},mark size=1pt] 
    table {
    x y
    5 5
    5 6
    5 7
    5 8
    5 9
    5 10
    5 11
    5 12
    5 13
    5 14
    5 15
    5 16
    5 17
    5 18
    5 19
    5 20
    5 21
    5 22
    5 23
    5 24
    5 25
    5 26
    5 27
    5 28
    5 29
    5 30
    5 31
    5 32
    6 5
    6 6
    6 7
    6 8
    6 9
    6 10
    6 11
    6 12
    6 13
    6 14
    6 15
    6 16
    6 17
    6 18
    6 19
    6 20
    6 21
    6 22
    6 23
    6 24
    6 25
    6 26
    6 27
    6 28
    6 29
    6 30
    6 31
    6 32
    7 5
    7 6
    7 7
    7 8
    7 9
    7 10
    7 11
    7 12
    7 13
    7 14
    7 15
    7 16
    7 17
    7 18
    7 19
    7 20
    7 21
    7 22
    7 23
    7 24
    8 5
    8 6
    8 7
    8 8
    8 9
    8 10
    8 11
    8 12
    8 13
    8 14
    8 15
    8 16
    8 17
    8 18
    8 19
    8 20
    9 5
    9 6
    9 7
    9 8
    9 9
    9 10
    9 11
    9 12
    9 13
    9 14
    9 15
    9 16
    9 17
    9 18
    10 5
    10 6
    10 7
    10 8
    10 9
    10 10
    10 11
    10 12
    10 13
    10 14
    10 15
    10 16
    10 17
    11 5
    11 6
    11 7
    11 8
    11 9
    11 10
    11 11
    11 12
    11 13
    11 14
    12 5
    12 6
    12 7
    12 8
    12 9
    12 10
    12 11
    12 12
    12 13
    13 5
    13 6
    13 7
    13 8
    13 9
    13 10
    13 11
    13 12
    14 5
    14 6
    14 7
    14 8
    14 9
    14 10
    14 11
    15 5
    15 6
    15 7
    15 8
    15 9
    15 10
    16 5
    16 6
    16 7
    16 8
    16 9
    16 10
    17 5
    17 6
    17 7
    17 8
    17 9
    17 10
    18 5
    18 6
    18 7
    18 8
    18 9
    19 5
    19 6
    19 7
    19 8
    20 5
    20 6
    20 7
    20 8
    21 5
    21 6
    21 7
    22 5
    22 6
    22 7
    23 5
    23 6
    23 7
    24 5
    24 6
    24 7
    25 5
    25 6
    26 5
    26 6
    27 5
    27 6
    28 5
    28 6
    29 5
    29 6
    30 5
    30 6
    31 5
    31 6
    32 5
    32 6
    };

\end{axis}
\end{tikzpicture}
\caption{Possible sizes of families $\aaa$, $\bb$ that span $\R^5$ and have binary scalar products.}
\label{posSizesPlot}
\end{figure}
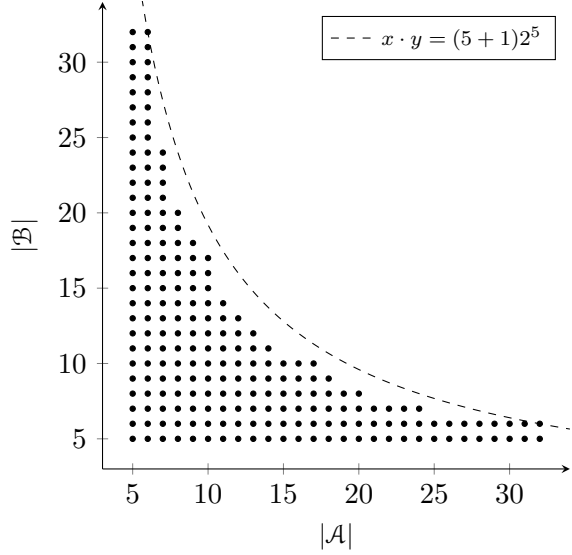

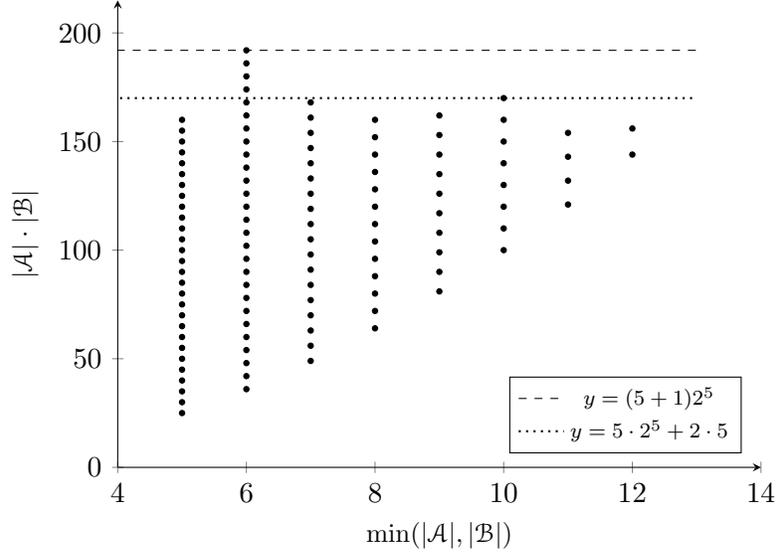
\begin{figure}[!h]
\centering
\begin{tikzpicture}
\begin{axis}[
    xlabel={$\operatorname{min}(|\aaa|,|\bb|)$},
    ylabel={$|\aaa|\cdot|\bb|$},
    axis lines=left,
    xmin=4, xmax=14,
    ymin=0, ymax=215, 
    legend pos=south east,
    legend style={font=\scriptsize},
    width=0.65\textwidth,
    height=0.5\textwidth,
    label style={font=\small}
]

\addplot[domain=3:13, samples=100, color=black, dashed] {192};
\addlegendentry{$y = (5+1)2^5$}

\addplot[domain=3:13, samples=100, color=black, line width=0.8pt, dotted] {170};
\addlegendentry{$y = 5\cdot2^5+2\cdot5$}

\addplot[only marks, color=black, mark=*, mark options={fill=black},mark size=1pt] 
    table {
    x y 
    5 25
    5 30
    5 35
    5 40
    5 45
    5 50
    5 55
    5 60
    5 65
    5 70
    5 75
    5 80
    5 85
    5 90
    5 95
    5 100
    5 105
    5 110
    5 115
    5 120
    5 125
    5 130
    5 135
    5 140
    5 145
    5 150
    5 155
    5 160
    6 36
    6 42
    6 48
    6 54
    6 60
    6 66
    6 72
    6 78
    6 84
    6 90
    6 96
    6 102
    6 108
    6 114
    6 120
    6 126
    6 132
    6 138
    6 144
    6 150
    6 156
    6 162
    6 168
    6 174
    6 180
    6 186
    6 192
    7 49
    7 56
    7 63
    7 70
    7 77
    7 84
    7 91
    7 98
    7 105
    7 112
    7 119
    7 126
    7 133
    7 140
    7 147
    7 154
    7 161
    7 168
    8 64
    8 72
    8 80
    8 88
    8 96
    8 104
    8 112
    8 120
    8 128
    8 136
    8 144
    8 152
    8 160
    9 81
    9 90
    9 99
    9 108
    9 117
    9 126
    9 135
    9 144
    9 153
    9 162
    10 100
    10 110
    10 120
    10 130
    10 140
    10 150
    10 160
    10 170
    11 121
    11 132
    11 143
    11 154
    12 144
    12 156
    };

\end{axis}
\end{tikzpicture}
\caption{$\operatorname{min}(|\aaa|, |\bb|)$ and $|\aaa||\bb|$ for families that span $\R^5$ and have binary scalar products.}
\label{pos-Min-Prod-Plot}
\end{figure}

The proof of Theorem~\ref{two_level_new_bound} makes use of Theorem~\ref{d2d_plus_2d} and a quick observation that $P$ is affinely isomorphic to the cube if it has too few facets to apply Theorem~\ref{d2d_plus_2d}. However, the case where $P$ has few vertices cannot be easily reduced to the case with few facets, as the set of possible pairs $\left(f_0(P), f_{d-1}(P)\right)$ for $2$-level $P$ is not symmetric. For example, a triangular prism in $\R^3$ is $2$-level, but there is no $3$-dimensional $2$-level $P$ satisfying $f_0(P)=5$ and $f_2(P)=6$.

\noindent We finish with  a conjecture that generalises Theorems~\ref{d_plus_one_two_d} and~\ref{d2d_plus_2d}.

\begin{conjecture}\label{generalisation}
    Let $\aaa,\bb \subseteq \R^d$ be families of vectors that both linearly span $\R^d$. Suppose that $\la a, b\ra \in \{0,1\}$ holds for all $a \in \aaa$, $b \in \bb$. Furthermore, suppose that $|\aaa|$ and $|\bb|$ are both strictly larger than $2^{k-1}(d-k+2)$ for some $k\in [0,d]$. Then $\left|\mathcal{A}\right| \cdot\left|\mathcal{B}\right| \leq (2^{d-k}+k)2^k(d-k+1)$.
\end{conjecture}
The motivating example for this conjecture is the following generalisation of Example~\ref{cubeOctop}:
\begin{example}\label{generalCubeOctop}
    Let $\{e_i\}$ be the standard basis of $\R^d$, $k\in [0,d]$,
    \begin{align*}
        \aaa= & \left\{\sum_{i=k+1}^{d}\delta_i e_i\right\}\cup\left\{e_1,\ldots,e_k\right\},\,\bb=\left\{\sum_{i=1}^{k}\delta_i e_i + e_j\right\}\cup\left\{\sum_{i=1}^{k}\delta_i e_i\right\}, \\
        & \text{ where } \delta_i\text{ range over }\left\{0, 1\right\} \text{ and }j\text{ over }[k+1,d].
    \end{align*}
    Here, $|\aaa|=2^{d-k}+k$ and $|\bb|=2^k (d-k+1)$.
\end{example}
We enumerated distinct sets with binary scalar products in dimensions up to 5, where `distinct' refers to an absence of linear isomorphism, and the results support Conjecture~\ref{generalisation}. The conjecture also holds for all sets that come from 2-level polytopes in dimensions up to 8.

\section{Preliminaries}\label{secStability}
\paragraph{Notation.}
In what follows, we will often treat vectors in $\R^d$ as points in an affine space, with $\operatorname{dim}$ always referring to the affine dimension while $\operatorname{span}$ referring to linear span. The set of integers  from $1$ to $n$ is denoted $[n]$. 

Let $\aaa, \bb$ be families of vectors that both linearly span $\R^d$ and have binary scalar products, that is, $\la a, b \ra \in \{0,1\}$ for all $a \in \aaa$ and $b \in \bb$. We will use the following two simple observations a few times throughout our proofs. Let $a_1, \ldots, a_d$ be a basis of $\R^d$ contained in $\aaa$. Consider the dual basis $a_1^*, \ldots, a_d^*$:

\[
    \la a_i, a_j^* \ra = 
    \begin{cases}
        1,\, i = j \\
        0,\, i\neq j
    \end{cases}
\]
and observe that elements of $\bb$ have $0/1$ coordinates when expressed in this dual basis, or, in other words, $\bb$ is a subset of what we would call a cube: 
\[
    \bb \subseteq \left\{\sum_{i=1}^d \delta_i a_i^*, \text{ where } \delta_i \text{ range over } \{0, 1\}\right\}.
\]
Another observation is that projecting one family on the linear span of a subset of another preserves the binary scalar products property: if $\aaa' \subseteq \aaa$ and $\pi_{\aaa'}: \R^d \rightarrow \operatorname{span}(\aaa')$ is the orthogonal projection, then
\[
    \forall a \in \aaa',\, b \in \bb:\: \la a, \pi_{\aaa'}(b) \ra = \la a, b \ra \in \{0,1\}. 
\]

\subsection{The setup from \cite{kupavskii22}}\label{secstability2} We will now introduce some notation and restate some claims proved in \cite{kupavskii22}. Proofs of those claims and inequalities are provided in the Appendix~\ref{appendix} for completeness.

\noindent Since we are interested in bounding the product $|\aaa||\bb|$ from above, we will assume that $\aaa$ and $\bb$ are inclusion-wise maximal with respect to the property of having binary scalar products and, in paritcular, $\zero \in \aaa,\bb$. Let $b_d \in \bb \setminus \{\zero\}$ be a vector with the maximum value of $\operatorname{max}(\operatorname{dim}\aaa_0,\operatorname{dim}\aaa_1)$, where 
\[
    \aaa_i=\left\{a \in \aaa:\left\langle a, b_d\right\rangle=i\right\} \text{ for } i=0, 1.
\]
The choice of $b_d$ among the vectors that maximise $\operatorname{max}(\operatorname{dim}\aaa_0,\operatorname{dim}\aaa_1)$, in cases where it is important, will be specified at a later stage. 
We denote the orthogonal projection onto $U=b_d^\bot$ by $\pi: \mathbb{R}^d \rightarrow U$. We say that $X\subset \R^d$ does not contain opposite points if  $\{x,-x\}\subseteq X$ is only possible if $x=\zero$. Below, we state the claims and inequalities from \cite{kupavskii22}.

\begin{restatable}{claim}{claimassumptions}
    \label{cl1}
    We may translate $\aaa$ and replace some points $b\in\bb$ by the opposites $-b$ such that the following properties hold.
    \begin{itemize}
        \item[(i)] We (still) have $\aaa = \aaa_0 \cup \aaa_1$, where $\aaa_i = \{ a \in \aaa : \la a, b_d\ra = i\}$ for $i=0,1$ such that
            \begin{equation} \label{eqa0gea1}
                |\aaa_0| \ge |\aaa_1|.
            \end{equation}
        \item[(ii)] We  have
            \begin{equation}
                \label{eqscala0}
                \la a, b\ra \in \{0,1\} \text{ for each } a \in \aaa_0 \text{ and } b \in \bb.
            \end{equation}
        \item[(iii)] The set $\pi(\bb)$ does not contain opposite points.
    \end{itemize}
\end{restatable}

\begin{restatable}{claim}{claimpreimagespi}\label{cl2}
   Every point in $\pi(\bb)$ has at most two preimages in $\bb$.
\end{restatable}

We denote the linear span of $\aaa_0$ by $U_0$ and define the orthogonal projection $\tau: U \rightarrow U_0$. Let $\bb_* \subseteq \bb$ be the set of $b \in \bb$ for which $\pi(b)$ has exactly one preimage under projection onto $U$.

\begin{restatable}{inequality}{ineqbasic}\label{in0}
   $\left|\aaa\right|\left|\bb\right| \leq 2\left|\aaa_0\right||\pi(\bb)|+\left|\aaa_1\right|\left|\bb \backslash \bb_*\right|$.
\end{restatable}

\begin{restatable}{claim}{claimpreimagestau}\label{cl3}
    $\left|\pi(\bb)\right| \leq 2^{d-1-\operatorname{dim} U_0}\left|\tau(\pi(\bb))\right|$.
\end{restatable}

\begin{restatable}{claim}{claimrestbbconstant}\label{cl4}
    $\bb \backslash \bb_* $ can be partitioned as $ \bb_0 \sqcup \bb_1$, with $\bb_0, \bb_1$ satisfying
    \[
        \forall b\in\bb_i: \left|\left\{\langle a, b\rangle: a \in \aaa_i\right\}\right|=1\text{ for }i=0, 1.
    \]
\end{restatable}

Due to Claim \ref{cl4} and \eqref{eqa0gea1}, the final term in Inequality \ref{in0} can be bounded as \[ |\aaa_1||\bb \backslash \bb_*| = |\aaa_1||\bb_0|+|\aaa_1||\bb_1| \leq |\aaa_0||\bb_0|+|\aaa_1||\bb_1|.\] Using this and applying Theorem~\ref{d_plus_one_two_d} to bound the first term in Inequality \ref{in0}, we obtain
\begin{restatable}{inequality}{ineqkey}\label{in1}
    $\left|\aaa\right| \cdot\left|\bb\right| \leq \left(\operatorname{dim} U_0+1\right) 2^d+\left|\aaa_0\right|\left|\bb_0\right|+\left|\aaa_1\right|\left|\bb_1\right|.$
\end{restatable}
\begin{restatable}{inequality}{ineqforclfive}\label{ineqForCl5}
    For $i = 0,1$ we have 
    \[
        |\aaa_i| \leq 2^{\operatorname{dim}(\aaa_i)},\;|\bb_i| \leq 2^{\operatorname{dim}(\operatorname{span}(\bb_i))} \text{, and } \operatorname{dim}(\aaa_i) + \operatorname{dim}\bigl(\operatorname{span}(\bb_i)\bigl) \leq d.
    \]
\end{restatable}

\begin{restatable}{claim}{claimaaaibbi}\label{cl5}
    For $i=0, 1$, we have $\left|\aaa_i\right|\left|\bb_i\right| \leq 2^d$.
\end{restatable}

\noindent Looking at the definition of $\bb_i$, we see that we can either assume  $\zero, b_d \in \bb_0$ or assume $\zero, b_d \in \bb_1$. Here and in what follows we assume that $\zero, b_d \in \bb_1$. Therefore, Claim \ref{cl5} actually implies
\begin{equation}\label{stronger5}
    \left|\aaa_1\right|\left|\bb_1\right| \leq 2^d,\; \left|\aaa_0\right|\left(\left|\bb_0\right|+2\right)\leq 2^d.
\end{equation}

\noindent Inequality~\ref{in1} an Claim~\ref{cl5} are used in~\cite{kupavskii22} to prove Theorem~\ref{d_plus_one_two_d} as follows.\\ 
If $\operatorname{dim}U_0 \leq d-2$, 
\begin{equation}\label{final-old-step-smalldim}
    |\aaa| |\bb| \leq (\operatorname{dim}U_0 + 1)2^d + |\aaa_0||\bb_0|+|\aaa_1||\bb_1| \leq (d-1)2^d + 2^d + 2^d = (d+1)2^d.
\end{equation}
If $\operatorname{dim}U_0 = d-1$, Inequality~\ref{ineqForCl5} and $\zero, b_d \in \bb_1$ implies $\bb_0 = \varnothing$, and we have
\begin{equation}\label{final-old-step-d-1}
    |\aaa| |\bb| \leq (\operatorname{dim}U_0 + 1)2^d + |\aaa_1||\bb_1| \leq d 2^d + 2^d = (d+1)2^d.
\end{equation}

\subsection{Auxiliary results}
In this section, we present several simple lemmas that are needed throughout the proofs. 
\begin{restatable}{inequality}{dfInequality}\label{in2}
    For an integer $2\leq f\leq d$, we have \[(d+f)(2^{d-1}+2^{d-f})\leq d2^d + 2d.\]
\end{restatable}

\begin{proof}
    We will prove this by induction on $d$: when $d=f$, the inequality holds with equality. Assuming that the statement is valid for $d$, let us verify it for  $d+1$. Denoting the left and right sides of the inequality as $l(d, f)$ and $r(d)$, respectively, we have
    \begin{equation*}
        \begin{split}
            r(d+1) - l(d+1, f) 
            & \geq \left(r(d+1) - r(d)\right) - \left(l(d+1, f) - l(d, f)\right) \\
            & = \left(d2^d + 2^{d+1} + 2\right) - \left(d+f+2\right)\left(2^{d-1}+2^{d-f}\right) \\
            & = 2^{d-f}\left(d-f+2\right)\left(2^{f-1}-1-\frac{2f}{d-f+2}\right)+2 \\
            & \geq 2^{d-f}\left(d-f+2\right)\left(2^{f-1}-1-f\right)
        \end{split}
    \end{equation*}
    The obtained expression is non-negative for $f>2$. For $f=2$ and $d\geq4$, we have $2^{f-1}-1-\frac{2f}{d-f+2}\geq0$, and for $f=2$ and $d=2,3$, the initial inequality can be checked explicitly.    
\end{proof}

\begin{restatable}[]{lemma}{sufficientFractionLemma}\label{78sufficientLemma}
    Let $S_1, S_2 \subseteq [d-1]$ be such that $|S_2 \setminus S_1| > 1$. Then the family 
    \[
        \bigl\{S \subseteq [d-1]:\: |S\cap S_2|-|S\cap S_1| \in \{-1, 0, 1\}\bigr\}
    \]
    contains at most $\frac{7}{8}\cdot 2^{d-1}$ sets. 
\end{restatable}

\begin{proof}
    We start by claiming that
    \begin{equation}\label{binom78}
        \forall n > 2,\,j\in \mathbb{Z}:\: \binom{n}{j-1}+\binom{n}{j}+\binom{n}{j+1} \leq \frac{7}{8}\cdot 2^{n}.
     \end{equation}
    This can be checked by induction on $n$: \eqref{binom78} holds for $n=3$, and assuming it holds for $n-1$, we have
    \[
        \tbinom{n}{j-1}+\tbinom{n}{j}+\tbinom{n}{j+1} = \Bigl(\tbinom{n-1}{j-2}+\tbinom{n-1}{j-1}+\tbinom{n-1}{j}\Bigr) + \Bigl(\tbinom{n-1}{j-1}+\tbinom{n-1}{j}+\tbinom{n-1}{j+1}\Bigr) \leq \tfrac{7}{8}\cdot 2^{n-1} + \tfrac{7}{8}\cdot 2^{n-1} = \tfrac{7}{8}\cdot 2^n.
    \]
    Denote $P = S_2 \setminus S_1$ and $Q = S_1 \setminus S_2$, with $q=|Q|$ and $p=|P|\geq 2$. Let us also denote 
    \[
        \mathcal{D}_j = \{T\subseteq P\cup Q: |T\cap P|-|T\cap Q|=j\}\text{ for an integer }j.
    \] 
    Clearly, $|S\cap S_2| - |S\cap S_1| = |S\cap P| - |S\cap Q|$, and it is sufficient to show that the family $\mathcal{D}_{-1}\cup \mathcal{D}_{0}\cup \mathcal{D}_{1}$ contains at most $\frac{7}{8}\cdot 2^{p+q}$ sets. This is obvious if $p=2$ and $q=0$, so we will further assume $p+q>2$. To any set $T\in \mathcal{D}_j$ we may assign the set $(T\cap P)\cup(Q\setminus T)$ of size $q+j$. Such assignment constitutes a bijection between $\mathcal{D}_j$ and $(q+j)$-subsets of $P\cup Q$. Thus, $|\mathcal{D}_j|=\binom{p+q}{q+j}$, and with \eqref{binom78} we conclude
    \[
        |\mathcal{D}_{-1}\cup \mathcal{D}_{0}\cup \mathcal{D}_{1}| = \binom{p+q}{q-1}+\binom{p+q}{q}+\binom{p+q}{q+1}\leq \frac{7}{8}\cdot 2^{p+q}.
    \]
\end{proof}

\begin{restatable}[]{lemma}{smallDifferenceLemma}\label{smallDifferenceLemma}
    Let $\sss$ be a family of subsets of $[d-1]$ such that $|\sss| = d$ and 
    \[
        \forall S_1, S_2 \in \sss:\:|S_2\setminus S_1|\leq 1.
    \]
    Then either $\sss = \{S\subseteq[d-1]: |S| \geq d-2\}$ or $\sss = \{S\subseteq[d-1]: |S| \leq 1\}$.
\end{restatable}

\begin{proof}
    The  statement is trivial for $d=2$, so in what follows we assume $d>2$. Then $|\sss|>2$ and clearly $\sss$ contains sets of at most two different sizes (that differ by one). Let $U, V \in \sss$ both be of size $k\in [d-2]$. Observe that there are now only four options for sets in $\sss$:
    \begin{enumerate}[(a)]
        \item $U \cup V$ of size $k+1$. \label{cup}
        \item Sets of size $k$ that are contained in $U \cup V$. \label{cupSubset}
        \item Sets of size $k$ that contain $U\cap V$ as a subset. \label{capSupset}
        \item $U \cap V$ of size $k-1$. \label{cap}
    \end{enumerate}
Since $|(U \cup V)\setminus (U \cap V)|=2$, the sets from \ref{cup} and \ref{cap} cannot occur simultaneously. Similarly, if sets $B, C$ satisfy \ref{cupSubset}, \ref{capSupset}, respectively, and both differ from $U$ and $V$, then $|B\setminus C|=2$. Thus, except for $U$ and $V$, the sets from \ref{cupSubset} and \ref{capSupset} cannot be present together. There are $k+1$ and $d-k$ sets satisfying \ref{cupSubset} and \ref{capSupset}, respectively, so $|\sss|=d$ is only possible if $k=d-2$ or $k=1$ with $\sss = \{S\subseteq[d-1]: |S| \geq d-2\}$ or $\sss = \{S\subseteq[d-1]: |S| \leq 1\}$, respectively. 
\end{proof}

\begin{restatable}[]{lemma}{crossPolyRepresentation}\label{crossPolyRepr}
    Let ${a_1, \ldots, a_{d-1}, v}$ be a basis of $\R^d$. Define
    \[
        s=v+\sum_{i=1}^{d-1}a_i\text{ and }P=\operatorname{Conv}\Bigl(\{\zero, a_1, \ldots, a_{d-1}\}\cup\{s, s-a_1, \ldots, s-a_{d-1}\}\Bigr).
    \]
    Then $P$ is affinely isomorphic to the cross-polytope.
\end{restatable}

\begin{proof}
    Let $\{e_i\}$ be the standard basis of $\R^d$ and consider the linear transform that takes $\{a_1, \ldots, a_{d-1}, s\}$ to $\{e_d + e_1, \ldots, e_d + e_{d-1}, 2e_d\}$. This transform and a translation by $-e_d$ maps $P$ to the standard cross-polytope 
    \[
        K = \operatorname{Conv}\Bigl(\{e_1, \ldots, e_d\}\cup\{-e_1, \ldots, -e_d\}\Bigr).
    \]
\end{proof}

\section{Proofs of the main results}
\subsection{Uniqueness for Theorem~\ref{d_plus_one_two_d}}\label{sec31}

We start by proving a result that characterizes configurations that attain equality  in Theorem~\ref{d_plus_one_two_d}. This can be considered a warm-up proof, which is then used as a carcass for the proof of Theorem~\ref{d2d_plus_2d}. We heavily rely on the notation and claims introduced in the previous section.

\begin{theorem}\label{uniq}
    Let $\aaa,\bb \subseteq \R^d$ both linearly span $\R^d$ such that $\la a, b\ra \in \{0,1\}$ holds for all $a \in \aaa$, $b \in \bb$. Then we only have $|\aaa| \cdot |\bb| = (d+1) 2^d$ if one of the families has size $d+1$ and the other is affinely isomorphic to $\{0,1\}^d$.
\end{theorem}
\begin{proof}
     Without loss of generality, we assume $\left|\aaa\right| \geq \left|\bb\right|$. We  use induction on $d$, the statement being obvious in dimension $1$. Assuming that the statement holds for smaller dimensions, we prove it in dimension $d$. Consider two options for $\operatorname{dim}U_0$.
    \begin{enumerate}
        \item $\operatorname{dim}U_0 \leq d-2$. From Inequality \ref{in1} and \eqref{stronger5}, we get:
        \begin{equation*}
            \left|\aaa\right| \cdot\left|\bb\right| \leq 
            \left(d - 1\right) 2^d + 2\cdot2^d - 2\left|\aaa_0\right| \leq
            \left(d + 1\right) 2^d - \left|\aaa\right| < \left(d + 1\right) 2^d.
        \end{equation*}
        \item $\operatorname{dim}U_0 = d-1$. Note that since $\zero \in \aaa_0$, the definition of $\bb_0$ implies $\bb_0 \subset U_0^\bot$, and thus  we have $\bb_0 = \varnothing$ (recall that $\zero, b_d \in \bb_1$). We consider two subcases:
        \begin{enumerate}
            \item[a)]\label{uniqfirstcase} $\bb_*\neq\varnothing$. As we see from \eqref{final-old-step-d-1}, equality in Theorem~\ref{d_plus_one_two_d} can only be achieved when 
            Inequality \ref{in1} (and consequently Inequality \ref{in0}) are tight, which is only the case when  $\left|\aaa_0\right|\left|\pi(\bb)\right|=d 2^{d-1}$ (and $\left|\aaa_0\right|=\left|\aaa_1\right|$). By the induction hypothesis, the former is possible in one of two cases:
            \begin{enumerate}
                \item[i)] $\aaa_0$ is affinely isomorphic to $\{0,1\}^{d-1}$. Then, $\left|\aaa\right|=\left|\aaa_0\right|+\left|\aaa_1\right|=2^d$, which is only possible if $\aaa$ is affinely isomorphic to $\{0,1\}^d$, and then $\bb$ can only consist of a basis and the zero vector.
                \item[ii)] $\left|\aaa_0\right|=d$. Then, since $\left|\bb\right|\leq\left|\aaa\right|=2d$,  $\left|\aaa\right|\cdot\left|\bb\right|\leq4d^2$, which is less than $(d+1)2^d$ for $d\geq4$. For $d=3$, the inequality $\left|\bb\right|\cdot\left|\aaa\right|\leq32$ cannot yield equality since $\left|\aaa\right|=6$. Finally, if $d=2$ then $|\aaa|=4$, thus $\aaa$ is affinely isomorphic to a square  and $|\aaa|\cdot|\bb| = 3 \cdot 2^2$ only if $|\bb|=3=d+1$.
            \end{enumerate}
            \item [b)] $\bb_*=\varnothing$. Then, $\bb_1 = \bb$ and, consequently, $\operatorname{dim}(\operatorname{span}(\bb_1))=d$. In this case Inequality \ref{ineqForCl5} implies $|\aaa_1|=1$.
            Similarly to case a), Inequality \ref{in0} is only tight in one of the following cases:
            \begin{enumerate}
                \item[i)]$\left|\aaa_0\right|=d$. Then, $|\aaa|\cdot|\bb| \leq |\aaa|^2 \leq \left(d+1\right)^2<\left(d+1\right)2^d$.
                \item[ii)]$\left|\aaa_0\right|=2^{d-1}\text{, }\left|\pi(\bb)\right|=d$. Then, $\left|\aaa\right|\cdot\left|\bb\right|=2d\left(2^{d-1}+1\right)$, which is less than $(d+1)2^d$ for $d>2$. For $d=2$, we have $\left|\aaa\right|\cdot\left|\bb\right|\leq\left|\aaa\right|^2=9<3\cdot2^2$.
            \end{enumerate}
        \end{enumerate}
    \end{enumerate}
\end{proof}

\subsection{Proof of Theorem~\ref{d2d_plus_2d}}\label{sec32}

For convenience, let us restate the theorem.
\mainth*
    As in the proof of Theorem \ref{uniq}, we  use induction on $d$, and without loss of generality assume that $\left|\aaa\right|\geq\left|\bb\right|$. Note that we can also assume that $\aaa$ and $\bb$ are inclusion-wise maximal with respect to the property of having binary scalar products. For $d<3$, the bounds in Theorems~\ref{d2d_plus_2d} and~\ref{d_plus_one_two_d} coincide. Assuming validity for smaller dimensions, let us prove the statement for dimension $d$. We consider cases depending on the  value of $\operatorname{dim}U_0$.

    \begin{enumerate}
        \item\label{caseVerySmallU0} $\operatorname{dim}U_0 < d-2$. Then, from Inequality \ref{in1} and Claim \ref{cl5}, we have:
        \begin{equation}\label{strongWhenLessd-2}
            \left|\aaa\right|\cdot\left|\bb\right| \leq \left(\operatorname{dim} U_0 + 1\right) 2^d + 2^d + 2^d\leq d 2^d.
        \end{equation}
        \item\label{caseSmallU0} $\operatorname{dim}U_0 = d-2$. Applying the induction hypothesis to the families $\tau(\pi(\bb))$ and $\aaa_0$, we have three cases:
        \begin{enumerate}
            \item[a)] $|\tau(\pi(\bb))|=d-1$. By maximality $\bb$ contained $\zero$, so $\tau(\pi(\bb))$ consists of zero and the basis of $U_0$. The maximality of $\aaa$ now implies that $\aaa_0$ is affinely isomorphic to $\{0,1\}^{d-2}$. From \eqref{stronger5}, it follows that $\left|\bb_0\right|\leq2$. 
            
            For a given $b\in\bb_0$, there are two vectors that project onto $\pi(b)$ under $\pi$. Since they have identical scalar products with all the vectors in $\aaa_0$, and in our considerations below we work with $\aaa_0$ only, we can assume $|\bb_0|$ is even: if one vector belong to $\bb_0,$ then we can w.l.o.g. assume that the second one belongs to $\bb_0$ as well. We thus have two scenarios:
            \begin{enumerate}
                \item[i)] $\left|\bb_0\right|=0$. Then, from Inequality \ref{in0} and Claim \ref{cl5}, we obtain:
                \begin{equation*}
                    \left|\aaa\right|\cdot\left|\bb\right|\leq4\left(d-1\right)2^{d-2}+2^d=d 2^d.
                \end{equation*}
                \item[ii)] $\left|\bb_0\right|=2$. Then $U_0^{\bot} \cap \bb$ consists of $\zero, b_d$ and two vectors from $\bb_0$. Let $\bb'$ be a subset of $\pi(\bb)$ containing all vectors $v$ such that $\tau(v)$ has two preimages in $\pi(\bb)$. Assume that $|\tau(\bb')|=k+1$ (and thus $|\bb'| = 2k+2$).  
                Among these $k+1$ vectors, let $t_2$ be the number of those vectors with both preimages in $\pi(\bb_1)$, and let $t_1+1$ be the number of those with exactly one preimage in $\pi(\bb_1)$ ($+1$ standning for the zero vector which belongs to  $\bb_1$). The remaining $k-t_1-t_2$ have both preimages in $\pi(\bb_*)$. Furthermore, let the vectors in $\tau(\pi(\bb))$ with a single preimage under $\tau$ consist of $q$ projections from $\pi(\bb_1)$ and $d-2-k-q$ projections from $\pi(\bb_*)$. Recall that by definition of $\bb_*$, $|\bb_*| = |\pi(\bb_*)|$, which consists of  $t_1 + 2(k-t_1-t_2)$ vectors in $\bb'$ (out of $t_1 + 1$ vectors mentioned above, all but $\zero\in\R^d$ have the second preimage in $\pi(\bb_*)$) and $d-2-k-q$ vectors in $\pi(\bb)\setminus \bb'$. Therefore, 
                \[
                    |\bb_*| = t_1 + 2(k-t_1-t_2) + (d-2-k-q) = k - t_1 - 2t_2 + d - 2 - q.
                \]
                Next, from definition and Claim~\ref{cl2}, $|\bb_1|=2|\pi(\bb_1)|$. Besides $\zero$, $\pi(\bb_1)$ consists of $2t_2+t_1$ vectors from $\bb'$ and $q$ vectors from $\pi(\bb)\setminus \bb'$. This means that
                \[
                    |\bb_1| = 2|\pi(\bb_1)| = 2(t_1 + 2t_2 + q + 1) = 2 + 4t_2 + 2t_1 + 2q.
                \]
                Adding this up, we have:

                \begin{equation*}
                    \begin{split}
                        \left|\bb\right|
                        & = \left|\bb_*\right|+\left|\bb_0\right|+\left|\bb_1\right| \\
                        & = (k-t_1-2t_2+d-2-q)+2+(2+4t_2+2t_1+2q) \\ 
                        & = d + k + q + t_1 + 2t_2 + 2.
                    \end{split}
                \end{equation*}
                First, consider the case when $t_2 > 0$. Then $\pi(\bb_1)$ contains two elements that differ by a vector orthogonal to $U_0$, thus $U_0^\bot\subset\operatorname{span}(\bb_1)$. Recall that $\tau(\pi(\bb))$ consists of zero and the basis of $U_0$, together with the previous observation this implies $\operatorname{dim}(\operatorname{span}(\bb_1)) = |\tau(\pi(\bb_1))|+1$. The family $\tau(\pi(\bb_1))$ consists of the zero vector,  $t_1 + t_2$ elements from $\tau(\bb')$ and $q$ elements from $\tau(\pi(\bb) \setminus \bb')$. We get that $\operatorname{dim}(\operatorname{span}(\bb_1)) = t_1+t_2+q+2$, and, by \eqref{ineqForCl5}, $|\aaa_1|\leq2^{d-t_1-t_2-q-2}$. Consequently,  $$|\aaa|= |\aaa_0| + |\aaa_1| \leq 2^{d-2} + 2^{d-2-t_1-t_2-q},$$ and we get the following chain of inequalities.
                 
                \begin{align}
                    \left|\aaa\right|\cdot\left|\bb\right| 
                    & \leq \left(2^{d-2} + 2^{d-2-t_1-t_2-q}\right)\left(d+k+q+t_1+2t_2+2\right) \nonumber \\ 
                    & \leq \left(2^{d-2} + 2^{d-2-t_1-t_2-q}\right)\left(2d+t_1+2t_2\right) \label{mid1} \\
                    & \leq \left(2^{d-1} + 2^{d-1-t_1-t_2-q}\right)\left(d+t_1+t_2\right) \nonumber \\
                    & \leq \left(2^{d-1} + 2^{d-1-t_1-t_2}\right)\left(d+t_1+t_2\right) \label{mid2} \\
                    & \leq \left(2^{d-1} + 2^{d-1-t_1-t_2}\right)\left(d+t_1+t_2+1\right) \nonumber \\
                    & \leq d2^d + 2d. \label{mid3}             
                \end{align}
          Here, the second inequality follows from $k+q\leq d-2$, and the last one follows from Inequality \ref{in2}. If $t_2=0$, we get a slightly weaker bound:
                \begin{equation*}
                    \operatorname{dim}(\operatorname{span}(\bb_1)) \geq t_1+t_2+q+1 =t_1+q+1. \end{equation*}
                With the same reasoning this means that \eqref{mid2} becomes $\left(2^{d-1} + 2^{d-t_1}\right)\left(d+t_1\right)$, which is still less than \eqref{mid3} when $t_1\geq2$ due to Inequality \ref{in2}. Finally, when $t_2=0$ and $t_1=0, 1$, expression \eqref{mid1} yields a bound by $d2^d$ and $(2^{d-2}+2^{d-3})(2d+1)=d 2^d - \left(d-\frac{3}{2}\right)2^{d-2}\leq d2^d$, respectively.
            \end{enumerate}
            \item[b)] $|\aaa_0|=d-1$. Then:
            \begin{equation*}
                |\aaa|\cdot|\bb|\leq|\aaa|^2\leq4|\aaa_0|^2\leq4(d-1)^2\leq d2^d+2d, 
            \end{equation*}
            valid for any $d\ge 1.$
            \item[c)] Both $|\aaa_0|$ and $|\tau(\pi(\bb))|$ are at least $d$. By induction this implies 
            \[
            \left|\aaa_0\right|\cdot\left|\tau(\pi(\bb))\right|\leq\left(d-2\right)\left(2^{d-2}+2\right).
            \]
            Using Inequality \ref{in0}, Claim \ref{cl3}, and \eqref{stronger5}, we have 
            \begin{equation*}
                \left|\aaa\right|\cdot\left|\bb\right| \leq 4\cdot\left(d-2\right)\left(2^{d-2}+2\right)+2 \cdot 2^d - 2\left|\aaa_0\right| = 2d(2^{d-1}+1)+2\left(3d-8-\left|\aaa_0\right|\right).
            \end{equation*}
            This proves \eqref{eqmainvec} when $\left|\aaa_0\right|\geq3d-8$. Otherwise,
            \begin{equation*}
                \left|\aaa\right|\cdot\left|\bb\right| \leq \left|\aaa\right|^2 \leq 4\left|\aaa_0\right|^2 \leq 4\left(3d-9\right)^2,
            \end{equation*}
            which is less than $d2^d+2d$ for $d\geq3$. 
        \end{enumerate}
        \item $\operatorname{dim}U_0 = d-1$. Again, applying the induction hypothesis to $\pi(\bb)$ and $\aaa_0$, we have three cases (recall that from the assumption $\zero,b_d\in\bb_1$, we have $\bb_0=\varnothing$):
        \begin{enumerate}
            \item[a)] $|\pi(\bb)|=d$, that is, $\pi(\bb)$ consists of zero and the basis of $U_0$, which by maximality of $\aaa$ means that $\aaa_0$ is isomorphic to $\{0,1\}^{d-1}$.
            \begin{enumerate}
                \item[i)] $\operatorname{dim}\bb_1=1$. In this case, $\bb_1=\{\zero, b_d\}$ and so $\left|\bb\right|=d+1$. This contradicts the condition $|\bb|\ge d+2$ in the statement of the theorem.
                \item[ii)] $\operatorname{dim}\bb_1=k\geq2$. Then $\left|\bb_1\right|=2k$ and $\left|\aaa_1\right|\leq2^{d-k}$ by Inequality \ref{ineqForCl5}. Thus, we have
                \begin{equation*}
                    \left|\aaa\right|\cdot\left|\bb\right|\leq(2^{d-1}+2^{d-k})(d+k) \leq d2^d+2d.
                \end{equation*}
                by Inequality \ref{in2}.
            \end{enumerate}
            \item[b)] $\left|\aaa_0\right|=d$. Then $|\aaa||\bb|\leq\left|\aaa\right|^2\leq4|\aaa_0|^2\leq4d^2$, which is not larger than $d2^d+2d$ for $d>3$. For $d=3$, $|\aaa|^2$ gives the desired bound when $|\aaa_1|\leq2$, and finally $|\aaa_1|=3$ would by Inequality \ref{ineqForCl5} imply 
            \[
                \operatorname{dim}\aaa_1=2 \Rightarrow |\bb_1|=2 \Rightarrow |\bb|\leq 5 \Rightarrow |\aaa|\cdot|\bb|\leq 3\cdot 2^3 + 2\cdot 3.
            \]
            \item[c)] Both $\left|\aaa_0\right|$ and $\left|\pi(\bb)\right|$ are at least $d+1$.
        \end{enumerate}
    \end{enumerate}
    The remainder of the proof deals with the case 3c). By the induction hypothesis, \begin{equation*}
            \left|\aaa_0\right|\cdot\left|\pi(\bb)\right|\leq\left(d-1\right)\left(2^{d-1}+2\right).
        \end{equation*}
        In the displayed chain below, we use Claim~\ref{cl2} in the first equality; in the fourth equality we use that $\bb_0 = \varnothing$, and thus $\bb\setminus \bb_*=\bb_1$; in the second inequality we use  Claim~\ref{cl5}.
        \begin{align}
            |\aaa|\cdot|\bb|& = (|\aaa_0|+|\aaa_1|)\cdot(2|\pi(\bb)|-|\bb_*|)\nonumber\\
            & = 2|\aaa_0||\pi(\bb)| + 2|\aaa_1||\pi(\bb)|- 2|\aaa_1||\bb_*| + |\aaa_1||\bb_*| - |\aaa_0||\bb_*|  \nonumber\\
            & = 2|\aaa_0||\pi(\bb)|  + |\aaa_1||\bb \setminus \bb_*| + |\aaa_1||\bb_*|- |\aaa_0||\bb_*| \nonumber\\
            & = 2|\aaa_0||\pi(\bb)|+|\aaa_1||\bb_1|-\left(|\aaa_0|-|\aaa_1|\right)|\bb_*| \label{general} \\
            & \leq 2\left(d-1\right)\left(2^{d-1}+2\right)+|\aaa_1||\bb_1|-\left(|\aaa_0|-|\aaa_1|\right)|\bb_*| \label{general_ind} \\
            & \leq 2\left(d-1\right)\left(2^{d-1}+2\right)+2^{d}-\left(|\aaa_0|-|\aaa_1|\right)|\bb_*| \nonumber \\
            & = d 2^{d} + 2d - \left(|\aaa_0|-|\aaa_1|\right)|\bb_*| + (2d-4). \label{2d_minus_4}
        \end{align}
        
        \noindent Thus, it suffices to show, for example, that 
        $\left(\left|\aaa_0\right|-\left|\aaa_1\right|\right)\left|\bb_*\right|\geq 2d-4$. \\
        \noindent First consider the case  $\operatorname{dim}\aaa_1=d-1$. Then $\bb_1=\{\zero, b_d\}$, and we get 
        
        \begin{align*}
|\aaa|\cdot|\bb|=&|\aaa||\pi(\bb)|+|\aaa|\cdot\frac{1}{2}|\bb_1|\leq 2|\aaa_0||\pi(\bb)| + |\aaa| \\ 
\leq& 2(d-1)\left(2^{d-1}+2\right)+ |\aaa| =  d 2^d  - 2^d +4d-4+ \left|\aaa\right|.
        \end{align*}
        Thus, \eqref{eqmainvec} holds when $\left|\aaa\right|\leq 2^d-2d+4 $. Note that $\left|\aaa\right|> 2^d-2d+4$ is indeed impossible, as that would imply $\left|\aaa_0\right|> 2^{d-1}-d+2$ and
        \begin{equation*}
            \left|\aaa_0\right|\cdot\left|\pi(\bb)\right| > (2^{d-1}-d+2)\cdot(d+1)\geq(d-1)(2^{d-1}+2),
        \end{equation*}
        which contradicts the induction hypothesis. 
        
    In what follows, we assume that $\operatorname{dim}\aaa_1<d-1 $. Let us show that, due to this, we can also assume that $|\aaa_0|>|\aaa_1|$. To this end, suppose we had $|\aaa_0|=|\aaa_1|$ and recall how Claim~\ref{cl1} gave us an opportunity to switch $\aaa_0$ and $\aaa_1$ places by translating $\aaa$ and replacing some points in $\bb$ by their opposites. Since $|\aaa_0|$ is not smaller then $|\aaa_1|$, this opportunity was not used, but since $|\aaa_0|=|\aaa_1|$, nothing stops us from employing this transform nevertheless. Since this has an effect of swapping $\aaa_0$ and $\aaa_1$, we reduce to a case where $\operatorname{dim}U_0<d-1$, for which the desired bound has been shown in cases \ref{caseVerySmallU0} and \ref{caseSmallU0}. In what follows, we assume $|\aaa_0| > |\aaa_1|$.

        Consider the orthogonal projection $\pi_{\bb_1}:\mathbb{R}^d\rightarrow\operatorname{span}(\bb_1)$. By the definition of $\aaa_1$, we have $\left|\pi_{\bb_1}(\aaa_1)\right|=1$. Let $k=\operatorname{dim(\operatorname{span}(\bb_1))}$. Since $\bb$ contains a basis of $\mathbb{R}^d$, we have
        \begin{equation}\label{A_dif_B_star}
            \left|\bb_*\right|\geq d-k ,\ 
            \left(\left|\aaa_0\right|-\left|\aaa_1\right|\right)\left|\bb_*\right| \geq d-k.
        \end{equation}
        We will now deal with possible values of $k$.
        \begin{enumerate}
            \item[i)] $k=1$, which means $\bb_1=\{0,\:b_d\}$. Since $\operatorname{dim}\aaa_1<d-1$, from Inequality~\ref{ineqForCl5} it follows that $\left|\aaa_1\right|\leq2^{d-2}$. Substituting this into \eqref{general_ind}, we obtain:
            \begin{equation*}
                |\aaa|\cdot|\bb|\leq d 2^d + 2d + (2d - 4 - 2^{d-1}) \leq d 2^d + 2d.
            \end{equation*}
            \item[ii)] $k=2$. From Inequality~\ref{ineqForCl5}, it follows that $\left|\bb_1\right|\leq4$, and $\left|\aaa_1\right|\leq2^{d-2}$. Due to \eqref{A_dif_B_star}, $\left|\bb_*\right|\geq d-2$, so if $\left|\aaa_0\right|-\left|\aaa_1\right|\geq 2$ then \eqref{2d_minus_4} yields the desired estimate. Similarly, \eqref{2d_minus_4} completes the proof if $\left|\aaa_0\right|-\left|\aaa_1\right|=1$ and $\left|\bb_*\right|\geq 2d-4$. Finally, if $\left|\aaa_0\right|-\left|\aaa_1\right|=1$ and $\left|\bb_*\right|<2d-4$, then:
            \begin{equation*}
                \left|\aaa\right|\cdot\left|\bb\right|=\left(2\left|\aaa_1\right|+1\right)\cdot\left(\left|\bb_*\right|+\left|\bb_1\right|\right)<\left(2^{d-1}+1\right)\cdot\left(2d-4+4\right)=d 2^d + 2d.
            \end{equation*}
            \item[iii)] $k=d$. Inequality \ref{ineqForCl5} implies that $\aaa_1$ consists of only one point. Hence, \eqref{general_ind} becomes
            \begin{equation*}
                \left|\aaa\right|\cdot\left|\bb\right|\leq 2\left(d-1\right)\left(2^{d-1}+2\right)+\left|\bb_1\right|-\left(\left|\aaa_0\right|-\left|\aaa_1\right|\right)\left|\bb_*\right| \leq 2\left(d-1\right)\left(2^{d-1}+2\right)+\left|\bb\right|,
            \end{equation*}
            which completes the proof when $|\bb| \leq 2^d-2d+4$. The opposite is indeed impossible, as it would contradict Theorem~\ref{d_plus_one_two_d}:
            \begin{equation*}
                \left|\aaa\right|\cdot\left|\bb\right|\geq\left|\bb\right|^2\geq\left(2^d-2d+4\right)^2>\left(d+1\right)2^d.
            \end{equation*}
        \end{enumerate}
\vskip+0.2cm
        Before proceeding with the last case in the proof, we need to analyze the structure of  $\aaa$. The family $\aaa$ contains zero, and
        \[
            \operatorname{span}\bigl(\pi_{\bb_1}(\aaa)\bigr)=\operatorname{span}\bigl(\pi_{\bb_1}(\operatorname{span}(\aaa))\bigr)=\operatorname{span}(\bb_1),
        \]
        which means $\pi_{\bb_1}(\aaa)$ contains $\zero$ and a basis of an $k$-dimensional space, so $|\pi_{\bb_1}(\aaa)|\geq k+1$. Assume that $|\pi_{\bb_1}(\aaa)|> k+1$. We have

        \begin{align}
            & \left|\bb_1\right|\cdot\left|\pi_{\bb_1}(\aaa)\right|\leq\left(k+1\right)2^k,\;\left|\pi_{\bb_1}(\aaa)\right|\geq k+2 \Rightarrow \left|\bb_1\right| \leq 2^k\left(1-\frac{1}{k+2}\right) \Rightarrow \nonumber \\
            & \left|\aaa_1\right|\left|\bb_1\right| \leq 2^d\left(1-\frac{1}{k+2}\right) \Rightarrow \left|\aaa\right|\cdot\left|\bb\right|\stackrel{\eqref{general_ind}}{\leq} d 2^d + 4d-4- \frac{2^d}{k+2}-(d-k). \nonumber
        \end{align}
        This proves \eqref{eqmainvec}, because for $d \geq 3$ and $k < d$
        \begin{equation*}
            d+k-4-\frac{2^d}{k+2}\leq2d-5-\frac{2^d}{d+1}=-\frac{1}{d+1}\left(2^d - (2d-5)(d+1)\right)\leq 0.
        \end{equation*}
In what follows, we thus may assume that $|\pi_{\bb_1}(\aaa)|=k+1$ and that $\pi_{\bb_1}(\aaa)$ consists of $\zero$ and of a basis of an $k$-dimensional space.
        
        \begin{enumerate}
            \item[iv)] $2 < k < d$. Note that, due to \eqref{A_dif_B_star}, $\bb_* \neq \varnothing$. Let's denote the elements of $\pi_{\bb_1}(\aaa)$ as $a_0 = 0, a_1, \ldots, a_k$, and their preimages in $\aaa$ as $\mathbb{A}_j=\pi_{\bb_1}^{-1}(a_j)\cap\aaa$. We'll choose the numbering such that $\mathbb{A}_1=\aaa_1$. Let $b_{11}, b_{12}, \ldots, b_{1k}$ be a basis of $\bb_1$ that is dual to $a_1, \ldots, a_k$. For example, according to our choice of numbering, $b_{11}=b_d$. Note that, due to $\bb$ being inclusion-wise maximal, all $b_{1j}$ must belong to $\bb_1$ (otherwise, they, along with $b_{1j}+b_d$ for $j>1$, could be added to $\bb$). If $\operatorname{dim}\aaa_1 < d-k$, we have $|\aaa_1||\bb_1|\leq 2^{d-1}$ and just like in part i), substitution into \eqref{general_ind} produces the desired estimate. Consequently, we can now assume that $\operatorname{dim}\aaa_1=d-k$.

            Our further plan is to write $\aaa$ in a particular basis to see that, due to $\operatorname{dim}\mathcal{A}_1=d-k$, any of the $b_{1j}$ could be initially chosen as $b_d$, and that a suitable choice of $b_d$ would lead to the desired bound.
            
            We will extend the basis $\{b_{11}, \ldots, b_{1k}\}$ with elements from $\bb_*$ to form a basis for $\mathbb{R}^d$ and represent $\aaa$ in the dual basis. Then vectors of $\aaa$, arranged as column-vectors, form a matrix of the following form:
            
            \begin{equation*} 
            \renewcommand{\arraystretch}{0.9}
            \aaa=\begin{pNiceArray}{w{c}{2pt}w{c}{2pt}w{c}{2pt}w{c}{2pt}w{c}{2pt}|w{c}{2pt}w{c}{2pt}w{c}{2pt}w{c}{2pt}w{c}{2pt}|w{c}{2pt}w{c}{2pt}w{c}{2pt}w{c}{2pt}w{c}{2pt}|w{c}{2pt}w{c}{2pt}w{c}{2pt}w{c}{2pt}w{c}{2pt}|w{c}{2pt}w{c}{2pt}w{c}{2pt}w{c}{2pt}w{c}{2pt}}[first-row,last-col]
                \Block{1-5}{\mathbb{A}_0}\phantom{0} & & & & &
                \Block{1-5}{\mathbb{A}_1} & & & & &
                \Block{1-5}{\mathbb{A}_2} & & & & &
                \Block{1-5}{\cdots} & & & & &
                \Block{1-5}{\mathbb{A}_k} & & & & &\\
                
                \Block{5-5}<\LARGE>{\bm{0}} & & & & &
                1 & 1 & \cdots & 1 & 1 & 
                0 & 0 & \cdots & 0 & 0 &
                \Block{10-5}<\Large>{\cdots} & & & & &
                \Block{4-5}<\Large>{\bm{0}} & & & & &\\
                
                  &  &  &  &  &
                \Block{4-5}<\Large>{\bm{0}} & & & & & 
                1 & 1 & \cdots & 1 & 1 &
                  &  &  &  &  & 
                  &  &  &  &  &\\

                  &  &  &  &  &
                  &  &  &  &  &
                \Block{3-5}<\Large>{\bm{0}} &  &  &  &  &
                  &  &  &  &  &
                  &  &  &  &  &  \ \:k\\
                  
                &&&&&&&&&&&&&&&&&&&&&&&&\\
                &&&&&&&&&&&&&&&&&&&& 1 & 1 & \cdots & 1 & 1 &\\ \hdottedline

                \phantom{0}&&&&&\Block[fill=[RGB]{235,235,235},respect-arraystretch]{5-5}{}&&&&&\phantom{0}&&&&&\phantom{0}&&&&&\phantom{0}&&&& &\\
                &\phantom{0}&&&&&\phantom{0}&&&&&\phantom{0}&&&&&\phantom{0}&&&&&\phantom{0}&&& &\\
                &&\phantom{0}&&&&&\phantom{0}&&&&&\phantom{0}&&&&&\phantom{0}&&&&&\phantom{0}&& &\ \:d-k\\
                &&&\phantom{0}&&&&&\phantom{0}&&&&&\phantom{0}&&&&&\phantom{0}&&&&&\phantom{0}& &\\
                &&&&\phantom{0}&&&&&\phantom{0}&&&&&\phantom{0}&&&&&\phantom{0}&&&&&\phantom{0} &\\
                
                \CodeAfter
                    \line[radius=0.7pt]{7-2}{9-4}
                    \line[radius=0.7pt]{7-7}{9-9}
                    \line[radius=0.7pt]{7-12}{9-14}
                    \line[radius=0.7pt]{7-22}{9-24}
                    \SubMatrix{.}{1-1}{5-25}{]}[right-xshift=7pt]
                    \SubMatrix{.}{6-1}{10-25}{]}[right-xshift=7pt]
                    \UnderBrace[left-shorten, yshift=5pt]{6-6}{10-10}{\substack{\operatorname{dim}=d-k
                    }
                    }
                \end{pNiceArray}
            \end{equation*}

            \bigskip\bigskip\bigskip  
            
            The affine dimenion of the highlighted block coincides with the affine dimension $\dim(\mathbb A_1)$ of $\mathbb{A}_1=\aaa_1$, which is $d-k$. There is therefore an affine basis of a $\R^d$ that consists of $d-k+1$ vectors from $\mathcal{A}_1$, one vector from each other $\mathbb{A}_l$, $l>1$, and  $\zero$ that belongs to $\mathbb{A}_0$. 
            We thus have
            \begin{equation*}
                \forall j>1\!:\;\operatorname{dim}(\aaa\cap b_{1j}^\bot) = \operatorname{dim}(\aaa\setminus\mathbb{A}_j) = d-1,
            \end{equation*}  
            which means that, indeed, any of the $b_{1j}$ could be set as $b_d$ from the start. Choose $b_{1j}$ with the smallest possible size of $\mathbb{A}_j$, and repeat all the same reasoning with $b_{1j}$ playing the role of $b_d$. Note that in this case, $\left|\aaa\setminus\mathbb{A}_j\right|>\left|\mathbb{A}_j\right|$, so there will be no need for translation of $\aaa$ that swaps $\aaa_0$ and $\aaa_1$ in Claim~\ref{cl1}. After this reassignment of $b_d$ and appropriate relabeling of families $\mathbb{A}_l$, we may  assume that $|\mathbb{A}_j|$, among positive $j$, is minimised by $j=1$.
            \begin{align}
                \forall j>1\!:\:|\mathbb{A}_1| \leq |\mathbb{A}_j|\; \Longrightarrow & \nonumber\\   |\aaa_0|-|\aaa_1| = \biggl(|\mathbb{A}_0|+\sum_{j>1}|\mathbb{A}_j|\biggr) - |\mathbb{A}_1| > \sum_{j>1}|\mathbb{A}_j| - |\mathbb{A}_1| & \geq(k-2)|\aaa_1| \geq |\aaa_1|.\label{dif_A1_bound}
            \end{align} 
            If $\left|\aaa_0\right|-\left|\aaa_1\right|\geq2d-4$, non-emptiness of $\bb_*$ and \eqref{2d_minus_4} imply the desired estimate. Otherwise
            \begin{align*}
                |\aaa_0|-|\aaa_1| \leq 2d-5 & \;\xRightarrow{\eqref{dif_A1_bound}}\; |\aaa_1| \leq 2d - 6\;\Longrightarrow |\aaa| = (|\aaa_0|-|\aaa_1|) + 2|\aaa_1| \leq 6d-17, \\
                &|\aaa|\cdot|\bb| \leq |\aaa|^2 \leq (6d-17)^2 < d 2^d + 2d,
            \end{align*}
            concluding the proof.
        \end{enumerate}

\subsection{Application to \textbf{2}-level polytopes} \label{sec2level}
For convenience, we restate our main result concerning two-level polytopes.
\twoLevelNew*

\begin{proof}[Proof of Theorem \ref{two_level_new_bound}]

For $d=2,3$ the desired bound coincides with the one in Theorem~\ref{two_level_old_bound}, so we will further assume $d>3$. Let us denote $V=f_0(P)$ and $F=f_{d-1}(P)$ for conciseness. Shift $P$ so that $0$ is a vertex of $P$. Let $\aaa$ denote the vertex set of $P$ and $\bb'$ denote the minimal set of vectors such that every facet of $P$ lies in a hyperplane $\{x\,:\,\langle x,b \rangle = \delta \}$ for some $\delta\in\{0,1\}$ and $b\in\bb'$. Let $\bb = \bb' \cup \{0\}$.
    If every vector in $\bb'$ defines one facet of $P$, we are done by Theorem \ref{d_plus_one_two_d}: 
    \[V\cdot F < \left|\aaa\right|\cdot\left|\bb\right| \leq (d+1)2^d<(d-1)2^{d+1}+8(d-1).\]
    Otherwise, let $b_d\in\bb'$ define two facets of $P$. If the facet $P\cap b_d^{\bot}$ contains less half of the vertices of $P$, shift $P$ again so that zero becomes a vertex from the other (parallel) facet, reintroduce families $\aaa$, $\bb$ as described above and select $b_d\in\bb'$ that now defines the same two facets. 
    Now, $P\cap b_d^{\bot}$ contains at least half of the vertices of $P$. Introduce $\aaa_i$, $\pi$ and $\bb_i$ as in Section~\ref{secStability}. Note that because we've ensured $|\aaa_0|\geq|\aaa_1|$, no transformations are required in Claim~\ref{cl1}, and we have $\la a, b\ra \in \{0,1\}$ for all $a\in\aaa$, $b\in\bb$.
    Since $\operatorname{dim}(\aaa_1)=d-1$, we have $\bb_0 = \emptyset, \bb_1=\{0,b_d\}$ and $\left|\pi(\bb)\right|=\left|\bb\right|-1$, which means 
    \begin{equation}\label{produc_when_fuldim}
        \left|\aaa\right|\cdot\left|\bb\right|=\left|\aaa_0\right|\cdot\left|\pi(\bb)\right|+\left|\aaa_1\right|\cdot\left|\pi(\bb)\right|+\left|\aaa\right|.
    \end{equation}
    Since every vector in $\bb'$ defines at most two facets of $P$ and also contains $\zero$, $F\le 2\left|\bb\right| -1$, and thus from \eqref{produc_when_fuldim} we conclude
    \begin{equation}\label{VF_straightforward}
        V\cdot F \leq 2 \left(\left|\aaa_0\right|\cdot\left|\pi(\bb)\right|+\left|\aaa_1\right|\cdot\left|\pi(\bb)\right|\right) \leq 4\cdot\left|\aaa_0\right|\cdot\left|\pi(\bb)\right|
    \end{equation}
    Consider three cases:
    \begin{enumerate}
        \item $\left|\aaa_0\right| > d$ and $\left|\pi(B)\right| > d$. By Theorem \ref{d2d_plus_2d}, we have
        \[\left|\aaa_0\right|\cdot\left|\pi(\bb)\right| \leq (d-1)2^{d-1}+2(d-1)\]
        and with \eqref{VF_straightforward} we are done. 
        \item $\left|\pi(B)\right| = d$. Together with $\bb_1=\{0,b_d\}$, this means that $\bb'$ is a basis of $\mathbb{R}^d$. Every vector in $\bb'$ then has to define two facets of $P$, since otherwise $P$ is unbounded. Thus $P$ is affinely isomorphic to the cube.
        \item $\left|\aaa_0\right| = d$. Note that as $|\aaa_1| \leq |\aaa_0|$ and $\operatorname{dim}(\aaa_1)=d-1$, we also have $\left|\aaa_1\right|=d$. If $|\pi(\bb)|\leq \frac{7}{8} \cdot 2^{d-1}$, then \eqref{VF_straightforward} implies $V \cdot F \leq \frac{7}{8} d \cdot 2^{d+1} < (d-1)2^{d+1} + 8(d-1)$, so we may further assume
        \begin{equation}\label{piBisLarge}
            |\pi(\bb)| > \frac{7}{8} \cdot 2^{d-1}.
        \end{equation}
        We will now make several observations about the structure of $\aaa$ and $\bb$ that will make it clear that $P$ is affinely isomorphic to the cross-polytope. Let $a_0=0, a_1, \ldots, a_{d-1}$ be the elements of $\aaa_0$ and $\{u_{1}, \ldots, u_{d-1}\}$ be the basis of $\operatorname{span}(\aaa_0)$, dual to $\{a_1, \ldots, a_{d-1}\}$. Note that for every $j\in [d-1]$ there is a facet of $P$ that contains vertices $\{a_0, \ldots, a_{d-1}\}\setminus \{a_j\}$ and differs from $\aaa_0$. The vector $b_{\{j\}}\in \bb$, orthogonal to this facet, must satisfy $\pi(b_{\{j\}})=u_j$. Given $S\subseteq [d-1]$, let us denote by $b_{S}$ an element of $\bb$ for which $\pi(b_S)=\sum_{j\in S}u_j$, if there is one, with $b_{\varnothing}=0$ to avoid ambiguity. Consider the basis of $\mathbb{R}^d$ that is dual to $\{b_{\{1\}}, b_{\{2\}}, \ldots, b_{\{d-1\}}, b_d\}$. It is $\{a_1, a_2, \ldots, a_{d-1}, v\}$ with $v$ that satisfies 
        \[
            \la v, b_d\ra = 1\text{ and }\forall j \in [d-1]:\:\la v, b_{\{j\}}\ra = 0.
        \]\
        This means that 
        \begin{equation}\label{A1-representation}
            \aaa_1 = \{v+\sum_{j \in S}a_j:\:S\in \sss\}
        \end{equation}
    
        for some family $\sss$ of subsets of $[d-1]$ with $|\sss| = d$. Our goal is to show that ${\sss= \{S\subseteq[d-1]:\:|S|\geq d-2\}}$, as then Lemma~\ref{crossPolyRepr} would imply that $P$ is affinely  isomorphic to the cross-polytope, and we would be done. For $T \subseteq [d-1]$ denote $\sigma_T=\sum_{j\in T} a_j$ and note that, given $b_S\in \bb$,
        \[
            \la \sigma_T, b_S \ra = \la \sigma_T, \pi(b_S) \ra = \Bigl< \sigma_T, \sum_{j\in S}\pi(b_{\{j\}}) \Bigr> = \Bigl< \sum_{j\in T} a_j, \sum_{j\in S}b_{\{j\}} \Bigr> = |T \cap S|.
        \]
        Arguing indirectly, assume that $\exists S_1, S_2 \in \sss:\:|S_2\setminus S_1| > 1$. Inequality \eqref{piBisLarge} and Lemma~\ref{78sufficientLemma} imply that there exists $b_S\in \bb$ such that $|S \cap S_2| - |S \cap S_1| > 1$. But \eqref{A1-representation}
        means that
        \begin{align*}
            \{-1,0,1\} & \ni \la v+\sum_{j \in S_2}  a_j\,,\,b_{S}\ra - \la v+\sum_{j \in S_1} a_j\,,\, b_{S} \ra = \la \sigma_{S_2} - \sigma_{S_1}\,,\, b_{S} \ra 
            = |S_2 \cap S| - |S_1 \cap S|,
        \end{align*}
         
        a contradiction. Therefore, $\forall S_1, S_2 \in \sss:\:|S_2\setminus S_1| \leq 1$, which by Lemma \ref{smallDifferenceLemma} implies that either $\sss = \{S\subseteq[d-1]: |S| \geq d-2\}$ or $\sss = \{S\subseteq[d-1]: |S| \leq 1\}$. In case of the former, \eqref{A1-representation} and Lemma~\ref{crossPolyRepr} imply that $P$ is affinely isomorphic to the cross-polytope, and we are done. Finally, assume, looking for a contradiction, that $\sss = \{S\subseteq[d-1]: |S| \leq 1\}$. Then \eqref{A1-representation} implies that $\aaa_1$ is simply $\aaa_0$ shifted by $v$. $P$ is therefore affinely isomorphic to the cartesian product of a segment with a \mbox{$(d-1)$-dimentional} simplex. Therefore, $P$ has $d+2$ facets and $|\pi(\bb)| = d+1 \leq \frac{7}{8} 2^{d-1}$ for $d>3$, contradicting \eqref{piBisLarge}.
    \end{enumerate}
\end{proof}

\bibliographystyle{abbrv}
\nocite{*}
\bibliography{references}

\appendix

\section{Appendix}\label{appendix}

This section is here in order to make the paper self-contained. Here, we essentially repeat the proofs from \cite{kupavskii22} of the claims that we formulated  in  Section~\ref{secstability2}. 

\claimassumptions*
\begin{proof}
    If $|\{ a \in \aaa : \la a, b_d\ra = 0\}| \le |\{ a \in \aaa : \la a, b_d\ra = 1\}|$, then we can choose any $a_* \in \aaa$ with $\la a_*, b_d\ra = 1$ (which exists since $\aaa$ spans $\R^d$) and replace $\aaa$ by $\aaa - a_*$, $\bb$ by $(\bb \setminus \{b_d\}) \cup \{-b_d\}$, and $b_d$ by $-b_d$.
    This yields~(i).

    After this replacement, for each $b \in \bb$ there is some $\eps_b \in \{\pm 1\}$ such that $\la a, b\ra \in \{0,\eps_b\}$ holds for all $a \in \aaa$.
    Each $b$ with $\{\la a,b\ra : a \in \aaa_0\} = \{0,-1\}$ is replaced by $-b$, which yields~(ii).

    Let $\aaa_1'$ be a translate of $\aaa_1$ such that $\zero \in \aaa_1'$.
    Note that, for each $b \in \bb$ we now have $\{\la a,b\ra : a \in \aaa_0\} = \{0,1\}$ or $\{\la a,b\ra : a \in \aaa_0\} = \{0\}$.
    In the second case, we replace $b$ by $-b$ if $\{\la a,b\ra : a \in \aaa_1'\} = \{0,-1\}$, otherwise we leave it as it is.

    It remains to show that $\pi(\bb)$ does not contain opposite points after this transformation.
    To this end, let $b,b' \in \bb$ such that $\pi(b) = \beta \pi(b')$ for some $\beta \ne 0$, where $\pi(b),\pi(b') \ne \zero$.
    We have to show that $\beta = 1$.
    Note that for every $a \in \aaa_0 \cup \aaa_1' \subseteq U$ we have
    \[
        \la a,b\ra = \la a, \pi(b)\ra = \beta \la a, \pi(b')\ra = \beta \la a,b'\ra.
    \]
    Suppose first that $\{\la a,b\ra : a \in \aaa_0\} \ne \{0\}$.
    By~\eqref{eqscala0} there exists some $a \in \aaa_0$ with $1 = \la a,b\ra = \beta \la a,b'\ra$.
    Thus, we have $\la a,b'\ra \ne 0$ and hence $\la a,b'\ra = 1$, again by~\eqref{eqscala0}.
    This yields $\beta = 1$.

    Suppose now that $\{\la a,b\ra : a \in \aaa_0\} = \{0\}$.
    Note that this implies $\{\la a,b'\ra : a \in \aaa_0\} = \{0\}$.
    As $\aaa_0 \cup \aaa_1'$ spans $U$, we must have $ \{\la a,b\ra : a \in \aaa_1'\} \ne \{0\}$ and hence there is some $a \in \aaa_1'$ with $\la a,b\ra = 1$.
    Moreover, we have $\beta \la a,b'\ra = 1$, and in particular $\la a,b'\ra \ne 0$.
    This implies $\la a,b'\ra = 1$ and hence $\beta = 1$.
\end{proof}

As in the previous proof, let $\aaa_1'$ be a translate of $\aaa_1$ such that $\zero \in \aaa_1'$.
Note that for each $b \in \bb$ there are $\eps_b,\gamma_b \in \{\pm 1\}$ such that
\begin{align}
    \label{eqscaleps}
    & \la a, b\ra \in \{0,\eps_b\} \text{ for each } a \in \aaa \text{ and} \\
    \label{eqscala1pgamma}
    & \la a, b\ra \in \{0,\gamma_b\} \text{ for each } a \in \aaa_1'.
\end{align}

\noindent The proofs of the subsequent claims rely on the following two lemmas.

\begin{lemma}
    \label{lemslice}
    Suppose that $X \subseteq \{0,1\}^d \cup \{0,-1\}^d$ does not contain opposite points.
    Then we have $|X| \le 2^{\dim X}$.
\end{lemma}
\begin{proof}
    We prove the statement by induction on $d \ge 1$, and observe that it is true for $d = 1$.
    Now let $d \ge 2$.
    If $\dim X = d$, then we are also done.
    It remains to consider to case where $X$ is contained in an affine hyperplane $H \subseteq \R^d$.
    Let $c = (c_1,\ldots,c_d) \in \R^d$, $\delta \in \{0,1\}$ such that $$H = \{ x \in \R^d : \la c,x\ra = \delta \}.$$
    For each $i \in \{1,\dots,d\}$ let $\pi_i : H \to \R^{d-1}$ denote the projection that forgets the $i$-th coordinate, and let $e_i \in \R^d$ denote the $i$-th standard unit vector. Note that $\pi_{i^*}(X) \subseteq \{0,1\}^{d-1} \cup \{0,-1\}^{d-1}$.

    Suppose there is some $i^* \in \{1,\dots,d\}$ such that $\la c, e_{i^*}\ra \ne 0$ and $\pi_{i^*}(X)$ does not contain opposite points.
    By the induction hypothesis we obtain
    \[
        |X| = |\pi_{i^*}(X)| \le 2^{\dim \pi_{i^*}(X)} = 2^{\dim X},
    \]
    where the first equality and the last equality hold since $\pi_{i^*}$ is injective (due to $\la c, e_{i^*}\ra \ne 0$).

    It remains to consider the case in which there is no such $i^*$.
    Consider any $i \in \{1,\dots,d\}$.
    If $\la c, e_i \ra \ne 0$, then there exist $x=(x_1,\ldots,x_d),x'=(x_1',\ldots,x_d') \in X$, $x \ne x'$ such that $\pi_i(x) = -\pi_i(x')$.
    We may assume that $\pi_i(x) \in \{0,1\}^{d-1}$ and hence $\pi_i(x') \in \{0,-1\}^{d-1}$.
    As $X$ does not contain opposite points, we must have $x_i = 1$ and $x'_i = 0$, or $x_i = 0$ and $x'_i = -1$.
    In the first case we obtain
    \begin{align*}
        2 \delta
        = \la c,x\ra + \la c,x'\ra
        & = [\la \pi_i(c), \pi_i(x)\ra + c_ix_i] + [\la \pi_i(c), \pi_i(x')\ra + c_ix'_i] \\
        & = [\la \pi_i(c), \pi_i(x)\ra + c_i] + [\la\pi_i(c), \pi_i(x')\ra] \\
        & = c_i.
    \end{align*}
    Similarly, in the second case we obtain $2 \delta = -c_i$.

    If $\delta = 0$, this would imply that $c = \zero$, a contradiction to the fact that $H \ne \R^d$.
    Otherwise, $\delta = 1$ and hence every nonzero coordinate of $c$ is $\pm 2$.
    Thus, for every $x \in \Z^d$ we see that $\la c,x\ra$ is an even number, in particular $\la c,x\ra \ne \delta$.
    This means that $X \subseteq \Z^d \cap H = \emptyset$, and we are done.
\end{proof}

\noindent A direct consequence of Lemma \ref{lemslice} that we will employ is

\begin{lemma}
    \label{lemsliceb}
    Let $\aaa,\bb \subseteq \R^d$ such that $\aaa$ spans $\R^d$, $\bb$ does not contain opposite points, and for every $b \in \bb$ there is some $\eps_b \in \{ \pm 1\}$ such that $\{\la a,b\ra : a \in \aaa\} \subseteq \{0,\eps_b\}$.
    Then we have $|\bb| \le 2^{\dim \bb}$.
\end{lemma}
\begin{proof}
    Let $a_1,\dots,a_d \in \aaa$ be a basis of $\R^d$ and express elements of $\bb$ in the dual basis, it then becomes a subset of $\{0,1\}^d \cup \{0,-1\}^d$ with no opposite points. By Lemma~\ref{lemslice}, $|\bb| \le 2^{\dim \bb}$.
\end{proof}

We are ready to continue with the proofs of the remaining claims.

\claimpreimagespi*
\begin{proof}
    Let $y := \pi(b)$ for some $b \in \bb$ and observe that $\pi^{-1}(y) = \{x \in \R^d : \pi(x) = y\}$ is a one-dimensional affine subspace.
    By~\eqref{eqscaleps} and Lemma~\ref{lemsliceb} we obtain $|\bb \cap \pi^{-1}(y)| \le 2$.
\end{proof}

\ineqbasic*
\begin{proof}
    Claim \ref{cl2} implies $|\bb| = 2|\pi(\bb)| - |\bb_*|$ or $2(|\pi(\bb)| - |\bb_*|) = |\bb \setminus \bb_*|$. With $|\aaa_0|\geq|\aaa_1|$ this gives
    \begin{align*}
        |\aaa||\bb| = (|\aaa_0| + |\aaa_1|)(2|\pi(\bb_*)| - |\bb_*|) &\leq 2|\aaa_0||\pi(\bb_*)| + 2|\aaa_1||\pi(\bb)|-2|\aaa_1||\pi(\bb)| \\
        &= 2|\aaa_0||\pi(\bb_*)| + |\aaa_1||\bb \setminus \bb_*|
    \end{align*}   
\end{proof}

\claimpreimagestau*
\begin{proof}
    Fix any $b \in \bb$ and let $v := \pi(b)$.
    Consider the orthogonal complement $W \subseteq U$ of $U_0$ in $U$.
    As $\tau^{-1}(\tau(v)) 
    = v + W 
    $, it suffices to show that
    \[
        |(v + W) \cap \pi(\bb)| \le 2^{d - 1 - \dim U_0}
    \]
    holds.
    To this end, consider the linear subspace $\Pi \subseteq U$ spanned by $v$ and $W$ and let $\sigma : U \to \Pi$ denote the orthogonal projection on $\Pi$.
    
    First, suppose that $\sigma(\aaa_1')$ spans $\Pi$.
    For every $a \in \aaa_1' \subseteq U$ and every $b \in \bb$ with $\pi(b) \in v + W \subseteq \Pi $ we have
    \[
        \la \sigma(a), \pi(b)\ra = \la a,\pi(b)\ra = \la a,b\ra \in \{0,\gamma_b\}
    \]
    by~\eqref{eqscala1pgamma}.
    Moreover, recall that $\pi(\bb)$ does not contain opposite points by Claim~\ref{cl1}~(iii).
    Thus, the pair $\sigma(\aaa_1')$ and $(v + W) \cap \pi(\bb)$ satisfies the requirements of Lemma~\ref{lemsliceb} (in $\Pi$), and hence we obtain
    \[
        |(v + W) \cap \pi(\bb)| \le 2^{\dim(v + W)} = 2^{\dim W} = 2^{\dim U - \dim U_0} = 2^{d - 1 - \dim U_0}.
    \]
    It remains to consider the case in which $\sigma(\aaa_1')$ does not span $\Pi$. Recall that we chose $b_d$ as the nonzero vector in $\bb$ with the maximal $\varphi(b_d):=\max\bigl(\operatorname{dim}(\aaa_0), \operatorname{dim}(\aaa_1)\bigr)$ for the corresponding $\aaa_0$ and $\aaa_1$. 
    Unless $|(v + W) \cap \pi(\bb)| = 1$, we will identify points $b_1,b_2 \in \bb$ with $\max \{ \varphi(b_1),\varphi(b_2) \} > \varphi(b_d)$, a contradiction to the choice of $b_d$.

    As $\aaa_0 \cup \aaa_1'$ spans $U$, we know that $\sigma(\aaa_0 \cup \aaa_1')$ spans $\Pi$.
    Since $\aaa_0$ is orthogonal to $W$, this means that $\sigma(\aaa_0)$ spans a line, and $\sigma(\aaa_1')$ spans a hyperplane $H$ in $\Pi$.
    Note that we have $v \notin W$ (otherwise $W = \Pi$ and so $\sigma(\aaa_1')$ spans $\Pi$).
    Thus, every nonzero point in $\sigma(\aaa_0)$ has nonzero scalar product with $v$.
    Moreover, for every $a \in \aaa_0$ with $\sigma(a) \ne \zero$ we have $\la \sigma(a), v\ra = \la a, v\ra = \la a,b\ra \in \{0,1\}$ by~\eqref{eqscala0}.
    Thus, since the nonzero vectors in $\sigma(\aaa_0)$ are collinear, we obtain
    \[
        \sigma(\aaa_0) \subseteq \{\zero, \sigma(a_0)\}
    \]
    for some $a_0 \in \aaa_0$. Since $\mathbf 0\in H,$ we have $\sigma(\aaa_0)\setminus H\subseteq \{\sigma(a_0)\}$ and further, since $\sigma(\aaa_0\cup \aaa_1')$ spans $\Pi$, we have $\sigma(\aaa_0)\setminus H= \{\sigma(a_0)\}$. 
    Let $c \in \Pi$ be a normal vector of $H$.  
    As $\sigma(a_0) \notin H$, we may scale $c$ so that $\la \sigma(a_0), c\ra = 1$.
    Let $a_* \in \aaa_1$ such that $\aaa_1' = \aaa_1 - a_*$.
    We define
    \[
        b_1 := c - \delta_1 b_d \ne \zero,
    \]
    where $\delta_1 := \la a_*, c\ra$.
    For every $a \in \aaa_0$ we have
    \[
        \la a,b_1\ra = \la a, c\ra = \la \sigma(a), c\ra \in \{\la \zero,c\ra, \la \sigma(a_0), c\ra\} = \{0,1\},
    \]
    and for every $a \in \aaa_1$ we have
    \begin{align*}
        \la a,b_1\ra
        = \la \underbrace{a - a_*}_{\in \aaa_1'}, b_1\ra + \la a_*, b_1\ra
        &= \la a - a_*, c\ra + \la a_*, b_1\ra
         = \la {\underbrace{\sigma(a - a_*)}_{\in H}}, c\ra + \la a_*, b_1\ra \\
        & = \la a_*, b_1\ra
        = \la a_*, c\ra - \delta_1 \la a_*, b_d\ra
        = \la a_*, c\ra - \delta_1
        = 0.
    \end{align*}
    Thus, by the maximality of $\bb$, (a scaling of) the vector $b_1$ is contained in $\bb$.
    Since we assumed $\zero \in \aaa_0$, we have $\varphi(b_1) \ge \dim(\aaa_1) + 1$.

    In order to construct $b_2$, let us suppose that there is another point $b' \in \bb$ with $v' := \pi(b') \ne v$ and $v' \in (v + W)$.
    If there is no such point, then the statement of the claim is true.
    Recall that $\sigma(a_0)$ is orthogonal to $W$, and let
    \[
        \xi := \la \sigma(a_0), v\ra = \la \sigma(a_0), \underbrace{v - v'}_{\in W}\ra + \la \sigma(a_0), v'\ra = \la \sigma(a_0), v'\ra.
    \]
    Choose $v'' \in \{v,v'\}$ such that $\xi c \ne v''$, and let $b'' \in \{b,b'\}$ such that $\pi(b'') = v''$.
    Define $\delta_2 := \la a_*,v'' - \xi c\ra$ and note that
    \[
        b_2 := v'' - \xi c - \delta_2 b_d
    \]
    is nonzero since $v'' - \xi c \in U \setminus \{\zero\}$.
    For every $a \in \aaa_0$ we have
    \[
        \la a,b_2\ra = \la a, \underbrace{v'' - \xi c}_{\in \Pi}\ra = \la \sigma(a), v'' - \xi c\ra,
    \]
    which is zero if $\sigma(a) = \zero$.
    Otherwise, $\sigma(a) = \sigma(a_0)$ and we obtain
    \[
        \la a,b_2\ra = \la \sigma(a_0), v''\ra - \xi \la \sigma(a_0), c\ra = \la \sigma(a_0), v''\ra - \xi = 0.
    \]
    Thus, $b_2$ is orthogonal to $\aaa_0$.
    Moreover, note that
    \[
        \la a_*, b_2\ra = \la a_*, v'' - \xi c\ra - \delta_2 \underbrace{\la a_*, b_d\ra}_{= 1} = 0.
    \]
    Thus, for every $a \in \aaa_1$ we have
    \begin{align*}
        \la a,b_2\ra
        = \la a - a_*, b_2\ra + \la a_*, b_2\ra
        = \la a - a_*, b_2\ra
        & = \la a - a_*, v''\ra - \xi \underbrace{\la a - a_*, c\ra}_{= 0} - \delta_2 \underbrace{\la a - a_*, b_d\ra}_{= 0} \\
        & = \la a - a_*, v''\ra = \la a - a_*, b''\ra \in \{0,\gamma_{b''}\}
    \end{align*}
    by~\eqref{eqscala1pgamma}.
    Thus, again by the maximality of $\bb$, (a scaling of) the vector $b_2$ is contained in $\bb$, and since $b_2$ is orthogonal to $\aaa_0$ and $a_* \in \aaa_1$, we have $\varphi(b_2) \ge \dim(\aaa_0) + 1$.
    However, by the choice of $b_d$ we must have
    \[
         \max \{ \dim(\aaa_0), \dim(\aaa_1) \} + 1 \le \max \{\varphi(b_1), \varphi(b_2)\} \le \varphi(b_d) = \max\{\dim(\aaa_0), \dim(\aaa_1)\},
    \]
    a contradiction.
\end{proof}

\claimrestbbconstant*
\begin{proof}
    Let $b \in \bb \setminus \bb_*$ and, for the sake of contradiction, suppose that $|\{ \la a,b\ra : a \in \aaa_0 \}| = |\{ \la a,b\ra : a \in \aaa_1 \}| = 2$.
    Let $b' \in \bb \setminus \{b\}$ such that $\pi(b) = \pi(b')$.
    In other words, we have $b' = b + \gamma b_d$ for some $\gamma \ne 0$.
    Then, by~\eqref{eqscala0} we have
    \[
        \{ \la a,b'\ra : a \in \aaa_0 \} = \{ \la a,b\ra : a \in \aaa_0 \} = \{0,1\}
    \]
    and hence we obtain $\eps_b = \eps_{b'} = 1$ by~\eqref{eqscaleps}.
    Again by~\eqref{eqscaleps} we see
    \[
        \{0,1\} \supseteq \{ \la a,b'\ra : a \in \aaa_1 \} = \{ \la a,b\ra : a \in \aaa_1 \} + \gamma = \{0,1\} + \gamma = \{\gamma, 1+\gamma\},
    \]
    which implies $\gamma = 0$, a contradiction.
\end{proof}

\ineqkey*
\begin{proof}
    $\tau(\pi(\bb))$ and $\aaa_0$ are both spanning $U_0$ and have binary scalar products, so by Theorem~\ref{d_plus_one_two_d} (or by the induction hypothesis, in the context of the proof of Theorem~\ref{d_plus_one_two_d} in \cite{kupavskii22})
    \[
        |\tau(\pi(\bb))||\aaa_0| \leq (\operatorname{dim}U_0+1)2^{\operatorname{dim}U_0}
    \]
    Combining this with Claim~\ref{cl3} and Inequality~\ref{in0} we get
    \[
        |\aaa||\bb| \leq 2 \cdot (\operatorname{dim}(U_0)+1)2^{d - 1} + |\aaa_1|(|\bb_0| + |\bb_1|) \leq \left(\operatorname{dim}U_0+1\right)2^d + |\aaa_0||\bb_0| + |\aaa_1||\bb_1|,
    \]
    where the second inequality is due to $|\aaa_0| \geq |\aaa_1|$.
\end{proof}

\ineqforclfive*
\begin{proof}
    The first (and second) inequality is a direct consequence of Lemma~\ref{lemsliceb} after writing $\aaa$ (or $\bb$) in the basis, dual to a basis found in $\bb$ (or $\aaa$). The last inequality follows from the definition of $\bb_i$: for each $b \in \bb_i$ there is $\xi_b$ such that 
    \[
        \aaa_i \subset W_i\text{, where }W_i = \{ x \in \R^d : \la x,b\ra = \xi_b \text{ for all } b \in \bb_i \},
    \] 
    and clearly $\operatorname{dim}(W_i) \leq d - \operatorname{dim}\bigl(\operatorname{span}(\bb_i)\bigr)$.
\end{proof}

\claimaaaibbi*
\begin{proof}
    By Inequality~\ref{ineqForCl5}, $|\aaa_i||\bb_i| \leq 2^{\operatorname{dim}(\aaa_i)} \cdot 2^{\operatorname{dim}(\operatorname{span}(\bb_i))} \leq 2^d$.
\end{proof}

\end{document}